\documentclass{amsart}

\usepackage{amssymb}
\usepackage{graphicx}
\usepackage{lineno}

\newtheorem{theorem}{Theorem}

\newtheorem{lemma}{Lemma}

\newcommand{\A}{{\mathcal A}}
\newcommand{\U}{{\mathcal U}}

\newcommand{\C}{{\mathbb C}}

\newcommand{\D}{{\mathbb D}}

\def\be{\begin{equation}}
\def\ee{\end{equation}}
\newcommand{\bthm}{\begin{theorem}}
\newcommand{\ethm}{\end{theorem}}
\newcommand{\beqq}{\begin{eqnarray*}}
\newcommand{\eeqq}{\end{eqnarray*}}

\begin{document}

\title[Sharp bounds of Hankel determinants for inverse functions]{Sharp bounds of Hankel determinants of second and third order for inverse functions of certain class of univalent functions}

\author[M. Obradovi\'{c}]{Milutin Obradovi\'{c}}
\address{Department of Mathematics,
Faculty of Civil Engineering, University of Belgrade,
Bulevar Kralja Aleksandra 73, 11000, Belgrade, Serbia}
\email{obrad@grf.bg.ac.rs}

\author[N. Tuneski]{Nikola Tuneski}
\address{Department of Mathematics and Informatics, Faculty of Mechanical Engineering, Ss. Cyril and Methodius
University in Skopje, Karpo\v{s} II b.b., 1000 Skopje, Republic of North Macedonia.}
\email{nikola.tuneski@mf.ukim.edu.mk}



\begin{abstract}
Let $\A$ be the class of functions that are analytic in the unit disc $\D$, normalized such that $f(z)=z+\sum_{n=2}^\infty a_nz^n$, and let class $\U(\lambda)$, $0<\lambda\le1$, consists of functions $f\in\A$, such that
\[ \left |\left (\frac{z}{f(z)} \right )^{2}f'(z)-1\right | < \lambda\quad (z\in\D). \]
In this paper we determine the sharp upper bounds for the Hankel determinants of second and third order for the inverse functions of functions from the class $\U(\lambda)$.
\end{abstract}

\maketitle

Let ${\mathcal A}$ be the class containing functions that are analytic in the unit disk $\D := \{ z\in \C:\, |z| < 1 \}$ and  are normalized such that
$f(0)=0= f'(0)-1$, i.e.,
\begin{equation}\label{eq-1}
f(z)=z+a_2z^2+a_3z^3+\cdots.
\end{equation}

\medskip

Further, let class  $\U(\lambda)$, $0<\lambda\le1$, consists  of all functions  $f\in {\mathcal A}$ such that
$$\left |\left (\frac{z}{f(z)} \right )^{2}f'(z)-1\right | < \lambda \quad (z\in \D).$$
The functions from the class $\U\equiv \U(1)$ are univalent (analytic and one-on-one), but are not starlike (doesn't map the unit disk onto a starlike region). The last is rare since the class of univalent starlike functions os very wide, and is main reason that class $\U(\lambda)$ attracts significant attention in past decades.

\medskip

More on this class can be found in \cite{OP-01,OP_2011,DTV-book}.

\medskip

Studying the behaviour of the coefficients of normalized univalent functions is one of the main directions for research in the theory of univalent functions. A problem that recently rediscovered, is to find upper bound (preferably sharp) of the modulus of the Hankel determinant $H_{q}(n)(f)$ of a given function $f$, for $q\geq 1$ and $n\geq 1$, defined by
\[
        H_{q}(n)(f) = \left |
        \begin{array}{cccc}
        a_{n} & a_{n+1}& \ldots& a_{n+q-1}\\
        a_{n+1}&a_{n+2}& \ldots& a_{n+q}\\
        \vdots&\vdots&~&\vdots \\
        a_{n+q-1}& a_{n+q}&\ldots&a_{n+2q-2}\\
        \end{array}
        \right |.
\]
The general Hankel determinant is hard to deal with, so the second and the third ones,
\[H_{2}(2)(f)= \left|\begin{array}{cc}
        a_2& a_3\\
        a_3& a_4
        \end{array}
        \right | = a_2a_4-a_{3}^2\] and
\[ H_3(1)(f) =  \left |
        \begin{array}{ccc}
        1 & a_2& a_3\\
        a_2 & a_3& a_4\\
        a_3 & a_4& a_5\\
        \end{array}
        \right | = a_3(a_2a_4-a_{3}^2)-a_4(a_4-a_2a_3)+a_5(a_3-a_2^2),
\]
respectively, are studied instead. The research is focused on the subclasses of univalent functions (starlike, convex, $\alpha$-convex, close-to-convex, spirallike,...) since the general class of normalised univalent functions is also hard to deal with. Some of the more significant results can be found in \cite{ckkls1,ckkls2,hayman-68,jhd1,jhd2,Kowalczyk-18,lss,kls,lrs,MONT-2019-2,zaprawa}.

\medskip

In this paper we will give the sharp upper bound of the modulus of the second and third Hankel determinant for the inverse functions for the functions in the class $\U(\lambda)$. Since functions $f$ from $\U(\lambda)$ are univalent, they have their inverse at least on the disk with radius 1/4 (due to the famous Koebe's 1/4 theorem). If the inverse has an expansion
\begin{equation}\label{eq-2}
f^{-1}(w) = w+A_2w^2+A_3w^3+\cdots,
\end{equation}
then, by using the identity $f(f^{-1}(w))=w$, from \eqref{eq-1} and \eqref{eq-2} we receive
\begin{equation}\label{eq 3}
\begin{array}{l}
A_{2}=-a_{2}, \\
A _{3}=-a_{3}+2a_{2}^{2} , \\
A_{4}= -a_{4}+5a_{2}a_{3}-5a_{2}^{3},\\
A_{5}= -a_{5}+6a_{2}a_{4}-21a_{2}^{2}a_3+3a_3^2+14a_2^4.
\end{array}
\end{equation}

\medskip

For our consideration we need the next lemma proven in \cite{MONT-2019-2}.

\medskip


 \begin{lemma}\label{lem1}
 For each function $f$ in $\mathcal{U}(\lambda),$ $0< \lambda \leq 1,$ there exists function $\omega _1,$ analytic in $\mathbb{D},$ such that $|\omega _1(z)|\leq |z| < 1,$ and $|\omega '_1(z)|\leq 1,$
 for all $z \in \mathbb{D},$ with
 \begin{equation}\label{eq1.1}
 \frac{z}{f(z)} =1-a_2z-\lambda z\omega _1(z).
\end{equation}
 Additionally, for
 \begin{equation}\label{eq-5}
  \omega _1(z)=c_1z+c_2z^2+\cdots,
   \end{equation}
we have
 \begin{equation}\label{eq-6}
 |c_1|\leq 1,\quad |c_2|\leq \frac{1}{2}(1-|c_1|^2)\quad   and \quad   |c_3|\leq \frac{1}{3}\left[1-|c_1|^2-\frac{4|c_2|^2}{1+|c_1|}\right].
 \end{equation}
\end{lemma}

\medskip

Using \eqref{eq1.1} and \eqref{eq-5} we have
  \begin{equation*}
z =[1-a_2z-\lambda z\omega _1(z)]f(z),
\end{equation*}
and after equating the coefficients,
\begin{align}\label{eq-7}
&a_3=\lambda c_1+a^2_2,\nonumber\\
&a_4=\lambda c_2+2\lambda a_2c_1+a_2^3,\\
\nonumber&a_5=\lambda c_3+2\lambda a_2c_2+\lambda ^2c_1^2+3\lambda a_2^2c_1+a_2^4,
\end{align}
that we will use later on.

\medskip

From \eqref{eq 3} and \eqref{eq-7}, after some calculations, we derive
\begin{equation}\label{eq-8}
\begin{array}{l}
A_{2}=-a_{2}, \\
A _{3}=-\lambda c_1+ a_{2}^{2} , \\
A_{4}= -\lambda c_2+3\lambda a_2c_1-a_2^3,\\
A_{5}= -\lambda c_3 + 4\lambda a_2c_2 -6\lambda a_2^2c_1 +2\lambda^2c_1^2+a_2^4.
\end{array}
\end{equation}

\medskip

We also need the next results from \cite{OPW_2016}.

\begin{lemma}\label{lem2}
Let $f\in\U(\lambda)$ for  $0<\lambda\le1$, and be given by $f(z)=z+\sum_{n=2}^{\infty}a_{n}z^n$. Then 
 \begin{equation}\label{eq-9}
  |a_2|\leq 1+\lambda.
 \end{equation}
 If $|a_2|=1+\lambda$, then $f$ must be of the form
\begin{equation}\label{ch-9-U-eq-003}
f(z) = \frac{z}{1-(1+\lambda)e^{i\phi}z+\lambda e^{2i\phi}z^2}
\end{equation}
for some $\phi\in[0,2\pi]$.
\end{lemma}

\medskip

In the same paper (\cite{OPW_2016}) it was conjectured that for functions in $\U(\lambda)$, $|a_n|\le\sum_{i=0}^{n-1}\lambda^i$ holds sharply, and was claimed to be proven in the case $n=3$:
\begin{equation}\label{eq-9-2}
|a_3|\leq 1+\lambda+\lambda^2.
\end{equation}
The proof rely on another claim, that for all functions $f$ from $\U(\lambda)$,
\begin{equation}\label{subord}
 \frac{f(z)}{z}\prec \frac{1}{(1+z)(1+\lambda z)}.
 \end{equation}
Recently, in \cite{lipon}, the second claim, and consequently the first one also, was proven to be wrong by giving a counterexample. Still, the subset of $\U(\lambda)$ when the inequality \eqref{eq-9-2} and subordination \eqref{subord} hold is nonempty, as the function 
\[  f_\lambda (z)= \frac{z}{(1-z)(1-\lambda z)}=\sum_{n=1}^{\infty}\frac{1-\lambda^n}{ 1-\lambda }z^n=z+(1+\lambda)z^2+(1+\lambda+\lambda^2)z^2+\cdots\]
shows. Here  $\left.\frac{1-\lambda^n}{ 1-\lambda }\right|_{\lambda =1}=n$ for all $n=1,2,3,\ldots$.

\medskip

Now we will give the sharp upper bound of the modulus of the second and the third Hankel determinant for the inverse functions of the functions from the class $\U(\lambda)$.

\begin{theorem}
Let $f\in\U(\lambda)$, $0<\lambda\le1$, and let its inverse is $f^{-1}$. Then
\begin{itemize} 
  \item[($i$)] $|H_2(2)(f^{-1})| \le \lambda(1+\lambda+\lambda^2)$ if the third coefficient of $f$ satisfies inequality \eqref{eq-9-2};
  \item[($ii$)] $|H_3(1)(f^{-1})| \le  \left\{
  \begin{array}{cc}
                   \frac{\lambda^2}{4}, & 0<\lambda\le \frac14, \\
                   \lambda^3, & \frac14\le\lambda\le1.
  \end{array}\right.$
\end{itemize}
Both results are sharp.
\end{theorem}

\begin{proof}
($i$) Using the definition of the second Hankel determinant, together with \eqref{eq-2} and \eqref{eq-8} we have
\[ H_2(2)(f^{-1}) = A_2A_4-A_3^2 = \lambda(a_2c_2-a_2^2c_1-\lambda c_1^2), \]
which gives
\[  |H_2(2)(f^{-1})| \le \lambda\left(|a_2||c_2|+ |c_1||\lambda c_1 +a_2^2|\right).\]
From \eqref{eq-9-2} we have $|\lambda c_1+a_2^2| = |a_3|\le 1+\lambda+\lambda^2$, from \eqref{eq-6}, $|c_2|\le \frac12(1-|c_1|^2)$, and so,
\begin{equation}\label{eq-11}
\begin{split}
  |H_2(2)(f^{-1})|
  &\le \lambda\left[\frac{|a_2|}{2}(1-|c_1|^2) + (1+\lambda +\lambda^2)|c_1|\right]\\
  &= \lambda\left[-\frac{|a_2|}{2}|c_1|^2  + (1+\lambda +\lambda^2)|c_1| + \frac{|a_2|}{2}\right] \\
  &:= \lambda \phi_1(|c_1|),
\end{split}
\end{equation}
where $\phi_1(t) = \lambda\left[-\frac{|a_2|}{2}t^2  + (1+\lambda +\lambda^2)t + \frac{|a_2|}{2}\right] $, $0\le t\le1$. Since for $f\in\U(\lambda)$ we have $|a_2|\le1+\lambda$ (see \eqref{eq-9}), then we conclude that $\phi_1$ is an increasing function on [0,1], i.e., it attains its maximum for $t=|c_1|=1$, which is equal to $1+\lambda+\lambda^2$. Using this and \eqref{eq-11} we have the statement of part (i) of this theorem. The result is sharp. Namely, for the function $f_{\lambda}(z) = \frac{z}{(1-z)(1-\lambda z)} = z+(1+\lambda)z^2+(1+\lambda+\lambda^2)z^3+ (1+\lambda+\lambda^2+\lambda^3)z^4+\cdots$, but using \eqref{eq 3} we have that
\[ f^{-1}_{\lambda}(w) = w- (1+\lambda)w^2 + (1+3\lambda+\lambda^2)w^3 - (1+6\lambda+6\lambda^2+\lambda^3)w^4+\cdots \]
and
\[H_2(2)(f^{-1}) = (1+\lambda)(1+6\lambda+6\lambda^2+\lambda^3)-(1+3\lambda+\lambda^2)^2 = \lambda(1+\lambda+\lambda^2).\]

\medskip

($ii$) Using \eqref{eq-2} and the definition of the third Hankel determinant, we have
\[ H_3(1)(f^{-1}) = A_3(A_2A_4-A_3^2) -A_4(A_4-A_2A_3)+A_5(A_3-A_2^2). \]
From this and \eqref{eq-8}, after some calculations we have
\[ \left| H_3(1)(f^{-1}) \right| = \left| \lambda^2c_1c_3 - \lambda^2c_2^2-\lambda^3c_1^3 \right| \le \lambda^2\left( |c_1||c_3| +|c_2|^2+\lambda|c_1|^3 \right), \]
and, by using the relations for $|c_2|$ and $|c_3|$ given in \eqref{eq-6}:
\begin{equation}\label{eq-12}
\begin{split}
  \left| H_3(1)(f^{-1}) \right|
  &\le \lambda^2\left[ |c_1|\frac13\left( 1-|c_1|^2-\frac{4|c_2|^2}{1+|c_1|}\right) +|c_2|^2 + \lambda|c_1|^3 \right]\\
  &= \frac{\lambda^2}{3} \cdot \left[|c_1|- |c_1|^3 +  \frac{3-|c_1|}{1+|c_1|}|c_2|^2 + 3\lambda|c_1|^3 \right]\\
  &\le \frac{\lambda^2}{3} \cdot \left[|c_1| +(3\lambda-1) |c_1|^3 +  \frac{3-|c_1|}{1+|c_1|}\frac14(1-|c_1|^2)^2 \right]\\
  &= \frac{\lambda^2}{12} \cdot \left[3-2|c_1|^2 + 12\lambda|c_1|^3 - |c_1|^4\right]\\
  &:= \frac{\lambda^2}{12} \cdot \phi_2(t),
  \end{split}
\end{equation}
where $\phi_2(t) = 3-2t^2 + 12\lambda t^3 - t^4,$ $0\le t=|c_1|\le1$.

\medskip

Since $\phi'_2(t) = -4t(t^2-9\lambda t+1)$, then if $0<\lambda\le\frac29$, we have $\phi'_2(t)\le0$ for $0\le t \le1$, and $\phi_2(t)\le\phi_2(0)=3$, which by \eqref{eq-12} implies $\left| H_3(1)(f^{-1}) \right|\le\frac{\lambda^2}{4}.$

\medskip

If $\frac29\le\lambda\le1$, then
\[
\begin{split}
\max\{ \phi_2(t):0\le t\le1\} &= \max\{\phi_2(0),\phi_2(1)\} = \max\{3,12\lambda\} \\
&= \left\{\begin{array}{cc}
                                                                                        3, & \frac29\le\lambda\le\frac14, \\
                                                                                        12\lambda, & \frac14\le\lambda\le1.
                                                                                      \end{array}\right.
                                                                                      \end{split}
                                                                                      \]

\medskip

Combining all this facts we receive the estimate ($ii$) of the theorem.

\medskip

The estimate ($ii$) is also sharp due to the functions $f_1(z) = \frac{z}{1-\frac{\lambda}{2}z^3} = z+\frac{\lambda}{2}z^4+\frac{\lambda^2}{4}z^7+\cdots$ and $f_2(z) = \frac{z}{1-\lambda z^2} = z+\lambda z^3+\lambda^2z^5+\cdots$. Indeed, it is easy to check that both functions are in $\U(\lambda)$. Further, for $f_1$ we have $a_2=a_3=a_5=0$, $a_4=\frac{\lambda}{2}$, and by \eqref{eq 3}, $A_2=A_3=A_5=0$, $A_4=-\frac{\lambda}{2}$, i.e., $ H_3(1)(f_1^{-1}) = -\frac{\lambda^2}{4}$ and $\left| H_3(1)(f_1^{-1}) \right| = \frac{\lambda^2}{4}$. Similarly, for the function $f_2$, $a_2=0$, $a_3=\lambda$, $a_4=0$, $a_5=\lambda^2$, and $A_2=0$, $A_3=-\lambda$, $A_4=0$. $A_5=2\lambda^2$, i.e.,  $\left| H_3(1)(f_2^{-1}) \right| = |-\lambda^3| = \lambda^3$.

\medskip

\end{proof}

\end{document}